\documentclass[11pt,a4paper]{amsart}
\usepackage{graphicx,multirow,array,amsmath,amssymb}
\usepackage{arydshln}

\newtheorem{theorem}{Theorem}
\newtheorem{proposition}[theorem]{Proposition}
\newtheorem{lemma}[theorem]{Lemma}
\newtheorem{corollary}[theorem]{Corollary}

\newtheorem*{example}{Example}

\begin{document}

\title[Fox $p$-Colorings as Fixed Points of Braid Representations]
{Fox $p$-Colorings as Fixed Points of Braid Representations}

\author[N. Goh]{Neungjoo Goh}
\address{Department of Mathematics Education, Sunchon National University, Sunchon 57922, Korea}
\email{rhsmdwn2912@gmail.com}

\author[S. Jeong]{Suah Jeong}
\address{Department of Mathematics Education, Sunchon National University, Sunchon 57922, Korea}
\email{wjtn1203@naver.com}

\author[S. Oh]{Seungbin Oh}
\address{Department of Mathematics Education, Sunchon National University, Sunchon 57922, Korea}
\email{osboyr5901@naver.com}

\author[H. Yoo]{Hyungkee Yoo}
\address{Department of Mathematics Education, Sunchon National University, Sunchon 57922, Korea}
\email{hyungkee@scnu.ac.kr}

\keywords{tricolorability, torus link, spiral link, braid}
\thanks{Mathematics Subject Classification 2020: 57K10}

\maketitle

\begin{abstract}
Fox $p$-colorings of knots and links admit a linear-algebraic
description in terms of braid representations.
For each braid $\beta \in B_n$ we associate a representation
$M : B_n \to \mathrm{GL}(n,\mathbb{F}_p)$
such that Fox $p$-colorings of the closed braid $\widehat{\beta}$
correspond to fixed points of $M(\beta)$.
As a consequence,
$$
\mathrm{Col}_p(\widehat{\beta})
=
p^{\dim \ker(M(\beta)-I)},
$$
yielding an explicit formula for the number of Fox $p$-colorings
of arbitrary knots and links.
We apply the method to torus links and spiral links,
obtaining explicit descriptions of their Fox 3-coloring spaces.
\end{abstract}

\section{Introduction}

Fox $n$-coloring is one of the most classical and accessible invariants
in knot theory \cite{CF, F}.
Given a diagram of a knot or link, a Fox $n$-coloring assigns to each arc
an element of $\mathbb{Z}_n$ so that at every crossing
$$
2x_{\mathrm{over}} \equiv x_{\mathrm{under}_1}
+ x_{\mathrm{under}_2} \pmod n.
$$
A coloring in which all arcs receive the same value is called trivial.

For a link $K$, we denote by
$$
\mathrm{Col}_n(K)
$$
the number of distinct Fox $n$-colorings of $K$.
It is well known that $\mathrm{Col}_n(K)$ is a link invariant.
The set of Fox $n$-colorings of a diagram forms a finite abelian group,
and Reidemeister moves induce natural isomorphisms between coloring groups.
In particular, constant colorings form a subgroup isomorphic to $\mathbb{Z}_n$,
so that
$$
n \mid \mathrm{Col}_n(K).
$$

Fox colorings admit an algebraic interpretation as representations
of the link group into the dihedral group $D_n$ \cite{F}.
When $K$ is a knot and $p$ is prime,
the determinant completely determines the number of Fox $p$-colorings:
$$
\mathrm{Col}_p(K)=p\,\gcd(p,\det(K)).
$$
Thus the existence of nontrivial colorings is governed by the Alexander polynomial.
However, while the determinant determines the cardinality,
it does not provide an explicit linear description of the coloring space itself.

It is natural to ask whether Fox $p$-colorings can be realized directly
as the fixed-point space of a concrete linear transformation of small dimension.
Classical diagrammatic matrices, such as the Goeritz matrix
\cite{Goe, Pr},
typically have size comparable to the number of arcs or crossings in a diagram.
In contrast, a braid representation of a link has dimension equal to the number of strands of the chosen braid,
which may remain fixed even as the crossing number grows.
This suggests that a braid-based formulation can provide a compact
and structurally transparent linear model for Fox colorings.

Every link can be represented as the closure of a braid \cite{A}.
To each braid generator we associate a local linear transformation
over the finite field $\mathbb{F}_p$ encoding the Fox coloring relation.
Extending multiplicatively yields a representation
$$
M \colon B_n \longrightarrow \mathrm{GL}(n,\mathbb{F}_p).
$$
The representation $M$ may be viewed as the unreduced Burau representation~\cite{Bu}
specialized at $t=-1$ and reduced modulo $p$.
Our emphasis, however, is not on representation theory itself,
but on the fixed-point interpretation that this specialization provides.

Under this construction, Fox $p$-colorings of the closed braid
$K=\widehat{\beta}$ correspond precisely to fixed points of the matrix $M(\beta)$.
Thus counting Fox $p$-colorings becomes a problem in elementary linear algebra:
we compute the dimension of the eigenspace of $M(\beta)$
corresponding to the eigenvalue $1$.
This reformulation replaces a system of diagrammatic congruences
by the study of a single linear operator.

Our main result is the following fixed-point formula.

\begin{theorem} \label{thm:main}
Let $K=\widehat{\beta}$ be the closure of a braid
$\beta\in B_n$.
Then
$$
\mathrm{Col}_p(K)
=
p^{\dim \ker(M(\beta)-I)}.
$$
\end{theorem}

This description applies uniformly to knots and links.
In subsequent sections, we develop this matrix formulation in detail
and apply it to families of links with structured braid descriptions,
including torus links and spiral links.

\section{A matrix representation associated to Fox colorings}

In this section we construct a matrix formulation of Fox $p$-colorings
using braid representations.
Throughout, braids are read from top to bottom.

Let $\beta\in B_n$ be a braid on $n$ strands.
Label the strands at the top of $\beta$ by a vector
$$
(a_1,a_2,\dots,a_n)\in \mathbb{F}_p^n.
$$
As the strands pass through the braid, these labels propagate
according to the Fox coloring rule at each crossing.
Since the coloring condition is linear, the labels at the bottom
are given by a vector
$$
(b_1,b_2,\dots,b_n)\in \mathbb{F}_p^n
$$
which depends linearly on the initial labels.
Thus each braid induces a linear transformation
$$
\mathbb{F}_p^n \longrightarrow \mathbb{F}_p^n.
$$

We now describe this transformation for the elementary braid generators.
Let $\sigma_i\in B_n$ be the standard generator corresponding to a crossing
between the $i$-th and $(i+1)$-st strands as drawn in Figure~\ref{fig:ele_braid}.

\begin{figure}[h!] 
\centering 
\includegraphics{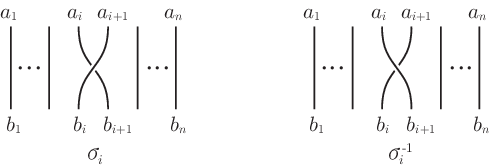} 
\caption{Elementary braids and the induced Fox coloring relations.} 
\label{fig:ele_braid} 
\end{figure}

If the labels immediately above the crossing are
$(a_1,\dots,a_n)$, then applying the Fox coloring condition gives
$$
b_i=a_{i+1}, \qquad
b_{i+1}=2a_{i+1}-a_i,
$$
and $b_j=a_j$ for $j\neq i,i+1$.
Equivalently,
$$
\begin{pmatrix}
b_1\\ \vdots\\ b_i\\ b_{i+1}\\ \vdots\\ b_n
\end{pmatrix}
=
\begin{pmatrix}
I_{i-1} & 0 & 0\\[2pt]
0 &
\begin{pmatrix}
0 & 1\\
-1 & 2
\end{pmatrix}
& 0\\[2pt]
0 & 0 & I_{n-i-1}
\end{pmatrix}
\begin{pmatrix}
a_1\\ \vdots\\ a_i\\ a_{i+1}\\ \vdots\\ a_n
\end{pmatrix}.
$$
We denote this matrix by $M(\sigma_i)\in \mathrm{GL}(n,\mathbb{F}_p)$.

Similarly, for the inverse generator $\sigma_i^{-1}$,
the Fox coloring rule yields
$$
b_i=2a_i-a_{i+1}, \qquad
b_{i+1}=a_i,
$$
with all other coordinates unchanged.
Hence
$$
\begin{pmatrix}
b_1\\ \vdots\\ b_i\\ b_{i+1}\\ \vdots\\ b_n
\end{pmatrix}
=
\begin{pmatrix}
I_{i-1} & 0 & 0\\[2pt]
0 &
\begin{pmatrix}
2 & -1\\
1 & 0
\end{pmatrix}
& 0\\[2pt]
0 & 0 & I_{n-i-1}
\end{pmatrix}
\begin{pmatrix}
a_1\\ \vdots\\ a_i\\ a_{i+1}\\ \vdots\\ a_n
\end{pmatrix},
$$
which we denote by $M(\sigma_i^{-1})$.
One verifies directly that $M(\sigma_i^{-1})=M(\sigma_i)^{-1}$.

The assignments $\sigma_i\mapsto M(\sigma_i)$
are compatible with the braid relations:
$$
M(\sigma_i)M(\sigma_{i+1})M(\sigma_i)
=
M(\sigma_{i+1})M(\sigma_i)M(\sigma_{i+1}),
$$
and $M(\sigma_i)$ commutes with $M(\sigma_j)$ whenever $|i-j|\ge 2$.
Thus the construction extends multiplicatively to a group homomorphism
$$
M \colon B_n \longrightarrow \mathrm{GL}(n,\mathbb{F}_p).
$$

If $\gamma$ is placed below $\beta$
so that strands pass through $\beta$ first and then through $\gamma$,
then the induced linear maps satisfy
$$
M(\beta\gamma)=M(\gamma)M(\beta).
$$
Equivalently, if $\beta=\sigma_{i_1}\sigma_{i_2}\cdots\sigma_{i_k}$ is read from
top to bottom, then
$$
M(\beta)=M(\sigma_{i_k})\cdots M(\sigma_{i_2})M(\sigma_{i_1}).
$$

Finally, observe that the vector $(1,1,\dots,1)\in \mathbb{F}_p^n$
is fixed by each $M(\sigma_i)$,
reflecting the trivial Fox coloring.
For a closed braid $\widehat{\beta}$,
a Fox $p$-coloring corresponds precisely to a vector
$a\in\mathbb{F}_p^n$ satisfying
$$
M(\beta)a=a,
$$
that is, to a vector in $\ker(M(\beta)-I)$.
This fixed-point interpretation forms the basis of
Theorem~\ref{thm:main}.

\section{Proof of the main results}

In this section we prove Theorem~\ref{thm:main} and its consequence,
Corollary~\ref{cor:power}, using the matrix representation $M$ constructed in the
previous section.
The key observation is that Fox $p$-colorings of a closed braid correspond exactly
to fixed points of the associated linear action.

\begin{proof}[Proof of Theorem~\ref{thm:main}]
Let $\beta\in B_n$ be a braid and label the strands at the top of $\beta$
by a vector $a=(a_1,\dots,a_n)\in\mathbb{F}_p^n$.
As described in Section~2, the Fox coloring condition at each crossing
propagates these labels linearly along the braid,
so that the labels at the bottom are given by $M(\beta)a$.

If $K=\widehat{\beta}$ denotes the closure of $\beta$,
then the top and bottom endpoints of each strand are identified.
Hence a Fox $p$-coloring of $K$ corresponds precisely to a vector
$a\in\mathbb{F}_p^n$ satisfying
$M(\beta)a=a$,
that is,
$$
(M(\beta)-I)a=0.
$$
Therefore the set of Fox $p$-colorings of $K$
is naturally identified with the kernel
$\ker(M(\beta)-I)$.

Since this kernel is a vector space over $\mathbb{F}_p$
of dimension $\dim_{\mathbb{F}_p}\ker(M(\beta)-I)$,
its cardinality is
$$
p^{\dim_{\mathbb{F}_p}\ker(M(\beta)-I)}.
$$
This proves
$$
\mathrm{Col}_p(K)
=
p^{\dim \ker(M(\beta)-I)}.
$$

Finally, the trivial coloring corresponds to the vector
$(1,1,\dots,1)\in\mathbb{F}_p^n$,
which is fixed by each $M(\sigma_i)$ and hence by $M(\beta)$.
Thus $K$ admits a nontrivial Fox $p$-coloring
if and only if
$$
\dim_{\mathbb{F}_p}\ker(M(\beta)-I)\ge 2.
$$
\end{proof}

Theorem~\ref{thm:main} shows that the entire space of Fox $p$-colorings
is determined by the fixed-point structure of the braid action.
This fixed-point interpretation allows immediate algebraic consequences.

An immediate algebraic consequence of Theorem~\ref{thm:main}
follows from the finiteness of $\mathrm{GL}(n,\mathbb{F}_p)$.

\begin{corollary}\label{cor:power}
For any braid $\beta\in B_n$ and any prime $p$, there exists a positive integer $m$
such that $M(\beta)^m=I$.
\end{corollary}

\begin{proof}
Fix a braid $\beta\in B_n$ and a prime $p$.
Since $\mathrm{GL}(n,\mathbb{F}_p)$ is a finite group, the element $M(\beta)$ has finite
order.
Hence there exists a positive integer $m$ such that
$$
M(\beta)^m=I.
$$
Equivalently, $M(\beta^m)=I$, and hence
$$
\ker(M(\beta^m)-I)=\mathbb{F}_p^n.
$$
The claim follows from Theorem~\ref{thm:main}.
\end{proof}

Corollary~\ref{cor:power} shows that for every braid $\beta$
there exists a positive integer $m$ such that the closure of $\beta^m$
admits the maximal possible number of Fox $p$--colorings.
Indeed, when $M(\beta)^m=I$ we obtain
$$
\ker(M(\beta^m)-I)=\mathbb{F}_p^n,
$$
so that
$$
\mathrm{Col}_p(\widehat{\beta^m})=p^n.
$$
The following corollary shows that the dimensions of the fixed-point spaces
$\ker(M(\beta)^m-I)$ are invariant under braid conjugation.

\begin{corollary}\label{cor:conj_kernel}
Let $\beta,\gamma\in B_n$ and let $p$ be a prime.
Then for every integer $m\ge1$,
$$
\dim\ker\!\bigl(M(\beta)^m-I\bigr)
=
\dim\ker\!\bigl(M((\gamma\beta\gamma^{-1})^m)-I\bigr).
$$
\end{corollary}

\begin{proof}
Set $A=M(\beta)$ and $P=M(\gamma)\in \mathrm{GL}(n,\mathbb{F}_p)$.
With the convention $M(\beta\gamma)=M(\gamma)M(\beta)$, we have
$$
M(\gamma\beta\gamma^{-1})
=
P^{-1}AP.
$$
Hence
$$
M(\gamma\beta\gamma^{-1})^m
=
(P^{-1}AP)^m
=
P^{-1}A^mP.
$$
Therefore
$$
M(\gamma\beta\gamma^{-1})^m-I
=
P^{-1}(A^m-I)P.
$$
Since $P$ is invertible, the matrices
$P^{-1}(A^m-I)P$ and $A^m-I$ are similar,
and hence their kernels have the same dimension.
\end{proof}

\section{Tricolorability of torus links}

Since tricolorability corresponds to the case $p=3$, it provides a natural and
computationally accessible specialization of the general theory developed above.
In this section we apply the braid-matrix criterion to study tricolorability for
several families of knots and links whose braid representations are particularly
well suited to explicit computations.

\begin{figure}[h!]
	\centering
	\includegraphics{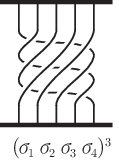}
	\caption{Torus link}
	\label{fig:torus}
\end{figure}

Recall that a \emph{torus knot} is a knot that lies on the surface of an unknotted
torus in $\mathbb{R}^3$.
More generally, a \emph{torus link} is a link embedded in the surface of a torus in
the same manner.
A torus link of type $T_{r,s}$ admits a standard braid representation of $r$ strings, while its crossing number grows with $|s|$.
This makes torus links a natural testing ground for braid-based coloring criteria.

For torus knots, the Alexander polynomial is well known,
and therefore the number of Fox $p$-colorings
can be computed directly from the determinant.
Hence the knot case is completely understood.
We therefore focus on torus links.
Recall that the torus link $T_{r,s}$
admits the standard braid representation
$$
T_{r,s}=\widehat{(\sigma_1 \sigma_2 \cdots \sigma_{r-1})^s}
\in B_r.
$$
Since $T_{r,s}=T_{s,r}$,
we may assume $1<r \le s$.
By Theorem~\ref{thm:main}, we have
$$
\mathrm{Col}_3(T_{r,s})
=
3^{\dim \ker(M((\sigma_1\cdots\sigma_{r-1})^s)-I)}.
$$
In particular,
$$
\mathrm{Col}_3(T_{r,s}) \le 3^r,
$$
since the number of strings is $r$.

We are interested in determining
for which values of $s$ equality holds,
that is,
$$
\mathrm{Col}_3(T_{r,s}) = 3^r.
$$
Equivalently, we give explicit exponents $s$
for which
$$
M((\sigma_1\cdots\sigma_{r-1})^s)=I.
$$

\begin{lemma}\label{lem:delta_matrix}
Let $\delta_r=\sigma_1\sigma_2\cdots\sigma_{r-1}\in B_r$
and let $A=M(\delta_r)\in\mathrm{GL}(r,\mathbb{F}_3)$.
Then $A$ has the following explicit form:

$$
A=
\begin{pmatrix}
0&1&0&\cdots&0\\
0&0&1&\cdots&0\\
\vdots&&&\ddots&\vdots\\
0&0&0&\cdots&1\\
\ast&\ast&\ast&\cdots&\ast
\end{pmatrix},
$$

where the first $r-1$ rows form the standard shift matrix, and the last row is given by
$$
A_{r,1}=(-1)^{r+1}, \qquad
A_{r,j}=2(-1)^{\,r-j} \quad (2\le j\le r)
$$
in $\mathbb{F}_3$.
\end{lemma}

\begin{proof}
Set $\delta_r=\sigma_1\sigma_2\cdots\sigma_{r-1}$.
With the convention in Section~2, we have
$$
A=M(\delta_r)=M(\sigma_{r-1})\cdots M(\sigma_2)M(\sigma_1).
$$
Recall that each $M(\sigma_i)$ replaces $(a_i,a_{i+1})$ by
$(a_{i+1},2a_{i+1}-a_i)$ and fixes the other coordinates.

Since each $M(\sigma_i)$ only modifies coordinates $i$ and $i+1$,
a straightforward induction on $r$ shows that the first $r-1$ rows of $A$
form the shift matrix
$$
A_{i,i+1}=1 \quad (1\le i\le r-1),
\qquad
A_{i,j}=0 \text{ otherwise for } i<r.
$$

The last row can be obtained by tracking the image of the basis vector
$e_r$ under successive applications of $M(\sigma_{r-1}),\dots,M(\sigma_1)$.
Each step introduces alternating signs and factors of $2$,
and a direct computation yields
$$
A_{r,1}=(-1)^{r+1}, \qquad
A_{r,j}=2(-1)^{\,r-j} \quad (2\le j\le r),
$$
as claimed.
\end{proof}

To illustrate the explicit form of the matrix $A=M(\delta_r)$,
we compute it directly for small values of $r$.

\begin{example}
For $r=3$, we have $\delta_3=\sigma_1\sigma_2$ and
$$
A=M(\sigma_2)M(\sigma_1)
=
\begin{pmatrix}
0&1&0\\
0&0&1\\
1&1&2
\end{pmatrix}
\in \mathrm{GL}(3,\mathbb{F}_3).
$$

For $r=4$, we obtain
$$
A=
\begin{pmatrix}
0&1&0&0\\
0&0&1&0\\
0&0&0&1\\
2&2&1&2
\end{pmatrix}
\in \mathrm{GL}(4,\mathbb{F}_3),
$$
which agrees with the general formula for the last row.
\end{example}

The above computations agree with the description in Lemma~\ref{lem:delta_matrix}.
We now make explicit the companion-matrix structure of $A$.

\begin{corollary}\label{cor:delta_companion}
With $A=M(\delta_r)$ as in Lemma~\ref{lem:delta_matrix},
the matrix $A$ is the companion matrix of a polynomial
$$
f_r(x)=x^r+c_{r-1}x^{r-1}+\cdots+c_1x+c_0
\in\mathbb{F}_3[x],
$$
where the coefficients are determined by
$$
-c_{j-1}=A_{r,j} \quad (1\le j\le r).
$$
In particular,
$$
f_r(A)=0.
$$
\end{corollary}

\begin{proof}
By Lemma~\ref{lem:delta_matrix},
the first $r-1$ rows of $A$ form the standard shift matrix.
Hence $A$ has the canonical companion form.

By definition of the companion matrix,
its characteristic polynomial equals
$$
x^r+c_{r-1}x^{r-1}+\cdots+c_0,
$$
where the coefficients satisfy
$$
-c_{j-1}=A_{r,j}.
$$
Therefore $A$ satisfies its defining polynomial,
that is, $f_r(A)=0$.
\end{proof}

The following theorem describes two infinite families
of torus links attaining the maximal possible number of
Fox $3$-colorings.

\begin{theorem}\label{thm:torus_tri}
Let $n$ be a positive integer.
\begin{enumerate}
\item For the torus link $T_{2n+1,\,4n+2}$,
$$
\mathrm{Col}_3\bigl(T_{2n+1,\,4n+2}\bigr)=3^{2n+1}.
$$
\item For the torus link $T_{2n,\,6n}$,
$$
\mathrm{Col}_3\bigl(T_{2n,\,6n}\bigr)=3^{2n}.
$$
\end{enumerate}
\end{theorem}

\begin{proof}
We now prove explicit periodicity identities for $A=M(\delta_r)$.
By Lemma~\ref{lem:delta_matrix}, the matrix $A$ has companion form, and
Corollary~\ref{cor:delta_companion} shows that $A$ satisfies its defining
polynomial
$$
f_r(x)=x^r+c_{r-1}x^{r-1}+\cdots+c_1x+c_0 \in \mathbb{F}_3[x],
\qquad\text{namely } f_r(A)=0,
$$
where the coefficients are determined by the last row of $A$ via
$-c_{j-1}=A_{r,j}$.
In particular, in each parity case below we first rewrite this explicit
$f_r(x)$ in a convenient closed form and then show that $f_r(x)$ divides
$x^N-1$ for a suitable exponent $N$.
Evaluating at $A$ then gives $A^N=I_r$.

\smallskip
\noindent\emph{(1) The case $r$ odd.}
Let $a_0=-1$ and $a_k=(-1)^k$ for $1\le k\le r-1$.
By Lemma~\ref{lem:delta_matrix} and Corollary~\ref{cor:delta_companion},
the defining polynomial of $A=M(\delta_r)$ is
$$
f(x)=x^r+\sum_{k=1}^{r-1}(-1)^k x^k-1\in\mathbb{F}_3[x],
$$
and hence $f(A)=0$.

Since $r$ is odd, we compute
$$
\sum_{k=1}^{r-1}(-1)^k x^k=\sum_{k=1}^{r-1}(-x)^k
=\frac{-x(1-x^{r-1})}{1+x}.
$$
Hence
$$
f(x)=x^r-1+\frac{-x(1-x^{r-1})}{1+x}
=\frac{(x^r-1)(x+1)-x+x^r}{x+1}
=\frac{x^{r+1}+2x^r-2x-1}{x+1}.
$$
Over $\mathbb{F}_3$ we have $2\equiv -1$, so
$$
x^{r+1}+2x^r-2x-1 \equiv x^{r+1}-x^r+x-1=(x^r+1)(x-1),
$$
and therefore
$$
f(x)=\frac{(x^r+1)(x-1)}{x+1}.
$$
Since $r$ is odd, we have $(-1)^r+1=0$ in $\mathbb{F}_3$, hence $x+1$ divides
$x^r+1$ and the above expression indeed lies in $\mathbb{F}_3[x]$.

Now
$$
x^{2r}-1=(x^r-1)(x^r+1),
$$
so
$$
\frac{x^{2r}-1}{f(x)}
=
\frac{(x^r-1)(x^r+1)}{(x^r+1)(x-1)/(x+1)}
=
\frac{(x^r-1)(x+1)}{x-1}.
$$
Since $x-1$ divides $x^r-1$ for all $r\ge 1$, the right-hand side lies in
$\mathbb{F}_3[x]$. Hence $f(x)\mid (x^{2r}-1)$ in $\mathbb{F}_3[x]$.

Therefore there exists $g(x)\in\mathbb{F}_3[x]$ such that
$$
x^{2r}-1=f(x)g(x).
$$
Evaluating at $A$ and using $f(A)=0$ yields
$$
A^{2r}-I = f(A)g(A)=0,
$$
so $A^{2r}=I_r$. Since $A=M(\delta_r)$, this proves $M(\delta_r)^{2r}=I_r$.
Taking $r=2n+1$ gives $s=2r=4n+2$, hence
$$
\mathrm{Col}_3\bigl(T_{2n+1,\,4n+2}\bigr)=3^{2n+1}.
$$

\smallskip
\noindent\emph{(2) The case $r$ even.}
Let $a_0=1$ and $a_k=(-1)^{k+1}$ for $1\le k\le r-1$.
By Lemma~\ref{lem:delta_matrix} and Corollary~\ref{cor:delta_companion},
the defining polynomial of $A=M(\delta_r)$ is
$$
f(x)=x^r+\sum_{k=1}^{r-1}(-1)^{k+1}x^k+1\in\mathbb{F}_3[x],
$$
and again $f(A)=0$.

Since $r$ is even, we compute
$$
\sum_{k=1}^{r-1}(-1)^{k+1}x^k
=\frac{x(1+x^{r-1})}{x+1},
$$
and hence
$$
f(x)=x^r+1+\frac{x(1+x^{r-1})}{x+1}
=\frac{(x^r+1)(x+1)+x+x^r}{x+1}
=\frac{x^{r+1}+2x^r+2x+1}{x+1}.
$$
Over $\mathbb{F}_3$ this becomes
$$
f(x)=\frac{x^{r+1}-x^r-x+1}{x+1}
=\frac{(x^r-1)(x-1)}{x+1}.
$$
Since $r$ is even, $(-1)^r-1=0$ in $\mathbb{F}_3$, so $x+1$ divides $x^r-1$
and the above expression is a polynomial in $\mathbb{F}_3[x]$.

In characteristic $3$ we have the Frobenius identity
$$
(x^r-1)^3=x^{3r}-1,
$$
so
$$
\frac{x^{3r}-1}{f(x)}
=
\frac{(x^r-1)^3}{(x^r-1)(x-1)/(x+1)}
=
\frac{(x^r-1)^2(x+1)}{x-1}\in\mathbb{F}_3[x],
$$
because $x-1$ divides $x^r-1$.
Thus $f(x)\mid (x^{3r}-1)$, and as in (1) we obtain $A^{3r}=I_r$.
Since $A=M(\delta_r)$, we have $M(\delta_r)^{3r}=I_r$.
Taking $r=2n$ gives $s=3r=6n$, hence
$$
\mathrm{Col}_3\bigl(T_{2n,\,6n}\bigr)=3^{2n}.
$$
This completes the proof.
\end{proof}

Theorem~\ref{thm:torus_tri} shows that maximal tricolorability occurs in infinite torus families
and can be detected via explicit periodicity of the braid action.
We next apply the same fixed-point method to spiral links.

\section{Tricolorability of spiral links}

We next consider a family of links defined by short periodic braid words.
Throughout this section we work over $\mathbb{F}_3$.

Let $r\ge 2$ and let
$$
\epsilon=(\epsilon_1,\epsilon_2,\dots,\epsilon_{r-1}),
\qquad
\epsilon_i\in\{\pm1\}.
$$
For an integer $s\ge 1$, form the braid
$$
\beta(r,s,\epsilon)
=
\left(
\sigma_1^{\epsilon_1}
\sigma_2^{\epsilon_2}
\cdots
\sigma_{r-1}^{\epsilon_{r-1}}
\right)^s
\in B_r.
$$
The closure of this braid,
denoted by
$$
S(r,s,\epsilon)=\widehat{\beta(r,s,\epsilon)},
$$
is called a \emph{spiral link}.

\begin{figure}[h!]
	\centering
	\includegraphics{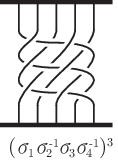}
	\caption{Spiral link}
	\label{fig:spiral}
\end{figure}

The spiral construction admits two natural symmetries on the sign vector
$\epsilon$.

\begin{proposition}
Let $\epsilon=(\epsilon_1,\dots,\epsilon_{r-1})$ and define

$$
-\epsilon=(-\epsilon_1,\dots,-\epsilon_{r-1}),
\qquad
\epsilon^{\mathrm{rev}}
=
(\epsilon_{r-1},\dots,\epsilon_1).
$$

Then the following hold:

\begin{enumerate}
\item[(1)] (Mirror symmetry)
$$
S(r,s,-\epsilon)
\cong
\overline{S(r,s,\epsilon)},
$$
where $\overline{S(r,s,\epsilon)}$ denotes the mirror image of $S(r,s,\epsilon)$.

\item[(2)] (Reversal symmetry)
$$
S(r,s,\epsilon^{\mathrm{rev}})
\cong
S(r,s,\epsilon).
$$
\end{enumerate}
\end{proposition}

\begin{proof}
(1) Replacing each exponent $\epsilon_i$ by $-\epsilon_i$
replaces every crossing $\sigma_i^{\pm1}$ by its inverse.
This corresponds to switching all crossings in the diagram,
which produces the mirror image of the link.
Hence
$$
S(r,s,-\epsilon)
\cong
\overline{S(r,s,\epsilon)}.
$$

(2) Consider the braid-group automorphism
$$
\varphi\colon B_r \to B_r,
\qquad
\varphi(\sigma_i)=\sigma_{r-i}.
$$
This automorphism is induced by a $180^\circ$ rotation of the punctured disk.
Applying $\varphi$ to the braid word gives
$$
\varphi(\beta(r,s,\epsilon))
=
\beta(r,s,\epsilon^{\mathrm{rev}}).
$$
Since such a rotation preserves braid closures up to isotopy,
we obtain
$$
S(r,s,\epsilon^{\mathrm{rev}})
\cong
S(r,s,\epsilon).
$$
\end{proof}

If $\epsilon=(1,1,\dots,1)$, then the braid
$\beta(r,s,\epsilon)$ coincides with $(\sigma_1\cdots\sigma_{r-1})^s$,
and hence $S(r,s,\epsilon)$ is the torus link $T_{r,s}$.
In the remainder of this section we assume that
$\epsilon\neq(1,1,\dots,1)$.

We begin with two explicit examples.

\begin{theorem}\label{thm:spiral_example1}
For any positive integer $k$, the spiral link $S(3,4k,(1,-1))$ is tricolorable.
\end{theorem}

\begin{proof}
Let $A=M(\sigma_1\sigma_2^{-1})\in\mathrm{GL}(3,\mathbb{F}_3)$.
A direct computation shows that $A^4=I$.
Hence
$$
M\bigl((\sigma_1\sigma_2^{-1})^{4k}\bigr)=A^{4k}=I.
$$
By Theorem~\ref{thm:main},
$$
\mathrm{Col}_3\bigl(S(3,4k,(1,-1))\bigr)
=
3^{\dim\ker(I-I)}
=
3^3.
$$
In particular, $S(3,4k,(1,-1))$ admits nontrivial tricolorings.
\end{proof}

\begin{theorem}\label{thm:spiral_example2}
For any positive integer $k$, the spiral links
$$
S(4,2k,(1,-1,1)),
\quad
S(4,3k,(1,-1,1)),
\quad
S(4,3k,(1,1,-1))
$$
are tricolorable.
\end{theorem}

\begin{proof}
We treat the three families separately.

\smallskip
\noindent\emph{(i) The links $S(4,2k,(1,-1,1))$ and $S(4,3k,(1,-1,1))$.}
Let
$$
w=\sigma_1\sigma_2^{-1}\sigma_3\in B_4,
\qquad
A=M(w)\in \mathrm{GL}(4,\mathbb{F}_3).
$$
With the convention $M(\beta\gamma)=M(\gamma)M(\beta)$, we have
$$
A=M(\sigma_3)\,M(\sigma_2^{-1})\,M(\sigma_1).
$$
A direct multiplication over $\mathbb{F}_3$ gives
$$
A=
\begin{pmatrix}
0&2&0&2\\
2&1&0&1\\
0&1&0&0\\
0&0&2&2
\end{pmatrix}.
$$
One checks further that
$$
A^6=I.
$$
In particular, $A^2$ has order $3$, so $A^{2k}\in\{I,A^2,A^4\}$ for every $k\ge 1$.
Moreover,
$$
A^2=
\begin{pmatrix}
1&2&1&0\\
2&2&2&1\\
2&1&0&1\\
0&2&1&1
\end{pmatrix},
\qquad
\dim\ker(A^2-I)=2.
$$
For instance, a row reduction shows that $\ker(A^2-I)$ is spanned by
$$
(1,0,0,1)^{\mathsf{T}},\qquad (0,1,1,0)^{\mathsf{T}}
$$
(over $\mathbb{F}_3$), so it is strictly larger than the $1$-dimensional space of
constant colorings. Hence $\dim\ker(A^{2k}-I)\ge 2$ for all $k\ge 1$, and therefore
$\widehat{w^{\,2k}}=S(4,2k,(1,-1,1))$ is tricolorable by Theorem~\ref{thm:main}.

Similarly, since $A^6=I$, we have $A^{3k}\in\{I,A^3\}$ for all $k\ge 1$.
A direct computation yields
$$
\dim\ker(A^3-I)=2,
$$
and hence $\dim\ker(A^{3k}-I)\ge 2$ for every $k\ge 1$.
Therefore $S(4,3k,(1,-1,1))=\widehat{w^{\,3k}}$ is tricolorable.

\smallskip
\noindent\emph{(ii) The links $S(4,3k,(1,1,-1))$.}
Let
$$
u=\sigma_1\sigma_2\sigma_3^{-1}\in B_4,
\qquad
B=M(u)\in \mathrm{GL}(4,\mathbb{F}_3).
$$
A direct computation gives
$$
B=
\begin{pmatrix}
0&2&0&2\\
2&0&2&1\\
0&1&0&0\\
0&0&1&2
\end{pmatrix}.
$$
One checks that $B^9=I$, and moreover
$$
\dim\ker(B^3-I)\ge 2
\quad
(\text{in fact, }\dim\ker(B^3-I)=3).
$$
For example, $(1,0,0,1)^{\mathsf{T}}$ and $(1,1,0,0)^{\mathsf{T}}$ lie in $\ker(B^3-I)$
and are linearly independent over $\mathbb{F}_3$.
Hence for every $k\ge 1$,
$$
M\!\bigl(u^{3k}\bigr)=B^{3k}=(B^3)^k
$$
has a fixed-point space of dimension at least $2$, so
$\widehat{u^{\,3k}}=S(4,3k,(1,1,-1))$ is tricolorable by Theorem~\ref{thm:main}.

This completes the proof.
\end{proof}

These examples illustrate that periodic behavior in the braid-matrix action can produce infinite families of tricolorable spiral links.
Even though the crossing number grows linearly in $s$, the dimension of the associated linear representation remains fixed, highlighting the computational efficiency of the braid-matrix formulation.

\bibliographystyle{abbrv}
\bibliography{references}

\end{document}